\documentclass[12pt,reqno,natbib,runningheads]{amsart}
\usepackage{graphicx}
\usepackage{amssymb}
\usepackage{amsfonts}
\usepackage{amsmath}
\usepackage{graphicx}

\numberwithin{equation}{section}
\numberwithin{figure}{section}
\numberwithin{table}{section}

\newtheorem{theorem}{Theorem}[section]
\newtheorem{corollary}{Corollary}[section]

\newtheorem{lemma}{Lemma}[section]
\newtheorem{example}{Example}[section]

\newtheorem{note}{Note}[section]

\begin{document}

\begin{center}
\textbf{\Large Coupled risk measures and their empirical estimation when losses follow heavy-tailed distributions} \bigskip

\bigskip

{\large
Abdelhakim Necir$^{\textrm{a}}$
and Ri\v{c}ardas Zitikis$^{\textrm{b},*}$}
\medskip

{$^{\textrm{a}}$\textit{Laboratory of Applied Mathematics,
Mohamed Khider University of Biskra, \break Biskra 07000, Algeria}
\smallskip

$^{\textrm{b}}$\textit{Department of Statistical and Actuarial Sciences,
University of Western Ontario,  \break
London, Ontario N6A5B7, Canada}
}

\end{center}

\bigskip

\begin{quote}
\noindent\textbf{Abstract.} Considerable literature has been devoted to developing statistical inferential results for risk measures, especially for those that are of the form of $L$-functionals. However, practical and theoretical considerations have highlighted quite a number of risk measures that are of the form of ratios, or even more complex combinations, of two $L$-functionals. In the present paper we call such combinations `coupled risk measures' and develop a statistical inferential theory for them when losses follow heavy-tailed distributions. Our theory implies -- at a stroke -- statistical inferential results for absolute and relative distortion risk measures, weighted premium calculation principles, as well as for many indices of economic inequality that have appeared in the econometric literature.
\medskip

\medskip

\end{quote}

\vfill

\noindent
{\small
$^{\textrm{*}}$Corresponding author:
Tel.: +1 519 432 7370;
fax.: +1 519 661 3813.
\\
E-mails:
\texttt{necirabdelhakim@yahoo.fr} (A.~Necir),
\texttt{zitikis@stats.uwo.ca} (R.~Zitikis)
}

\newpage

\section{\bf Introduction}
\label{sect-1}

Risk measures or premium calculation principles are functionals $\Pi: \mathcal{F} \to [0,+\infty]$ from the set $\mathcal{F}$ of all `loss' cumulative distribution functions (cdfs) $F$  to the extended non-negative real line $[0,+\infty]$ (e.g., Denuit et al. 2005). The corresponding random variables  $X \sim F$ are called `loss' random variables, and we assume them to be non-negative throughout this paper. We also assume that the variables $X$ follow, or are modeled with, continuous cdfs. Given a cdf $F$, the corresponding quantile function $Q:(0,1)\to [0,\infty )$ is defined by
\[
Q(t)=\inf\{x:\, F(x)\ge t\} .
\]
The quantile function $Q$ plays a pivotal role in defining numerous risk measures, and is a well known risk measure itself, called the value at risk and denoted by $\textrm{VaR}[t,F]$ (e.g., Denuit et al. 2005). Other illustrative examples of risk measures follow next.

\begin{example}\label{ex-1.1}\rm
Let $g:[0,1]\to [0,1]$ be a distortion function, that is, a non-decreasing function such that $g(0)=0$ and $g(1)=1$. The distortion risk measure (Denneberg, 1994; Wang, 1995, 1998) is defined by the formula
\[
\Pi_{d}[F]= \int_0^{\infty }g(1-F(x))dx,
\]
which can be rewritten in terms of the quantile function $Q$ as follows:
\begin{equation}
\Pi_{d}[F]=\int_{0}^{1} Q(s)d\Psi(s),
\label{pcp-dist}
\end{equation}
where $\Psi(s)=-g(1-s)$. Assuming that $g$ is left-continuous, the function $\Psi $ is right-continuous, and Note \ref{nt-1} below will clarify our reason for imposing this type of continuity. Hence, $\Pi_{d}[F]$ is an $L$-functional, a property that was highlighted and utilized by Jones and Zitikis (2003) and subsequently used by many researchers for developing statistical inferential results for distortion risk measures in light- and heavy-tailed settings.
\end{example}

\begin{note}\label{nt-1}\rm
To make this paper less cumbersome, we always write integrals as $\int_a^b g(x)dh(x)$ irrespectively of whether integrators $h$ are right- or left-continuous functions. What we have in mind behind such integrals is
$\int_{(a,b]} g(x)dh(x)$ in the case of right-continuous $h$ and $\int_{[a,b)} g(x)dh(x)$ in the case of left-continuous $h$. For example, the integral $\int_{0}^{1} Q(s)d\Psi(s)$ on the right-hand side of equation (\ref{pcp-dist}) means $\int_{(0,1]} Q(s)d\Psi(s)$ because $\Psi $ is right-continuous as noted below equation (\ref{pcp-dist}).
\end{note}

\begin{example}[Continuation of Example \ref{ex-1.1}]\label{ex-1.1b}\rm
Since $g$ is non-decreasing, the function $\Psi $ is also non-decreasing. In the special case when
\begin{equation}
\Psi(s)=-(1-s)^{1/\rho}
\label{psi-pht}
\end{equation}
with parameter $\rho \ge 1$, the distortion risk measure $\Pi_{d}[F]$ is called the proportional hazards transform and frequently denoted by $\textrm{PHT}[\rho,F]$. When
\begin{equation}
\Psi(s)={(s-t)_{+} \over 1-t }
\label{psi-cte}
\end{equation}
with parameter $t \in [0,1)$, then the risk measure $\Pi_{d}[F]$ is known as the tail value at risk and frequently denoted by  $\textrm{TVaR}[t,F]$. (We have used the classical notation $(s-t)_{+}$ for the positive part of $s-t$.) Since the cdf $F$ is continuous throughout this paper by assumption, the risk measure $\textrm{TVaR}[t,F]$ coincides with the conditional tail expectation  $\textrm{CTE}[t,F]=\mathbf{E}[X|X>Q(t)]$. For details on these and other risk measures, we refer to, e.g., Denuit et al. (2005) and references therein. This concludes Example \ref{ex-1.1}.
\end{example}

\begin{example}\label{ex-1.2a}\rm
The relative distortion risk measure (Wang, 1998) is given by the formula
\[
\Pi_{rd}[F]= {\Pi_{d}[F] \over \mathbf{E}[X]}.
\]
 The measure can be expressed (see equations of Example \ref{ex-1.1}) in terms of the quantile function $Q$ as follows:
\begin{equation}
\Pi_{rd}[F]={\int_{0}^{1} Q(s)d\Psi(s) \over \int_{0}^{1} Q(s)ds } .
\label{pcp-reldist}
\end{equation}
This is a ratio of $L$-functionals. With this note, we conclude Example \ref{ex-1.2a}.
\end{example}

\begin{example}\label{ex-1.2b}\rm
Let $w:[0,\infty )\to [0,\infty )$ be a non-decreasing function, called weight function. The weighted risk measure, or the weighted premium calculation principle (Furman and Zitikis, 2008), is given by the formula
\[
\Pi_{w}[F]={\frac{\mathbf{E}[Xw(X)] }{\mathbf{E}[w(X)]}}.
\]
For details on $\Pi_{w}[F]$ and its extensions, we refer to Furman and Zitikis (2009). With $U$ denoting a uniform on $[0,1]$ random variable, the variables $X$ and $Q(U)$ are equal in distribution, and so we can rewrite $\Pi_{w}[F]$ in terms of the quantile function $Q$ as follows:
\begin{equation}
\Pi_{w}[F]
= {\frac{\int_{0}^{1} H_{1}\circ Q(s)ds }{\int_{0}^{1} H_{2}\circ Q(s)ds}},
\label{pcp-weig}
\end{equation}
where $H_{1}(x)=xw(x)$ and $H_{2}(x)=w(x)$ with $H_{i}\circ Q(s)$
denoting the composition $H_{i}(Q(s))$ of the two functions $H_{i}$ and $Q$.
This concludes Example \ref{ex-1.2b}.
\end{example}

All of the aforementioned risk measures are special cases of the risk measure
\begin{equation}
\Pi_{r}[F]={\int_{0}^{1} H_{1}\circ Q(s)d\Psi_1(s) \over \int_{0}^{1} H_{2}\circ Q(s)d\Psi_2(s)  },
\label{gen-ratio}
\end{equation}
where $H_i:[0,\infty )\to [0,\infty ) $, $i=1,2$, are two non-decreasing and  left-continuous functions, and $\Psi_i:[0,1]\to \mathbf{R} $, $i=1,2$, are two non-decreasing and right-continuous functions. In addition to the aforementioned (absolute) PHT and CTE/TVaR risk measures, the ratio risk measure $\Pi_{r}[F]$ also includes the relative PHT risk measure $\int_0^{\infty }(1-F(x))^{1/\rho }dx / \mathbf{E}[X]$ and the relative CTE/TVaR risk measure $\mathbf{E}[X|X>Q(t)]/ \mathbf{E}[X]$. Many indices of economic inequality are also of form (\ref{gen-ratio}), as elucidated by Greselin et al. (2009). Another example of $\Pi_{r}[F]$ will follow after a note.

\begin{note}\rm
Some applications might lead to non-monotonic functions $H_{i}^{\circ}$. In this case we need to assume that the functions are of bounded variation, and to also require that each of the two non-decreasing components $H_{i}^{\ast}$ and $H_{i}^{\ast\ast}$ in the decomposition $H_{i}^{\circ}=H_{i}^{\ast}-H_{i}^{\ast\ast}$ satisfies the conditions to be imposed on $H_{i}$ in our following considerations. Same arguments apply to non-monotonic functions $\Psi_{i}$, assuming in particular that they are of bounded variation.
\end{note}

\begin{example}\label{ex-zenga}\rm
The risk measure $\Pi_{r}[F]$ includes, as a special case, the ratio
\[
R_F(p)
={\mathbf{E}[X|X \le Q(p)] \over \mathbf{E}[X|X > Q(p)] },
\]
which defines the Zenga curve $Z_F(p)$, $0\le p \le 1$, via the equation $Z_F(p)=1-R_F(p)$. We refer to the original papers by Zenga (1987, 2007) for interpretations and other details related to $Z_F$, and to Greselin et al. (2010) for statistical inferential results when (income) distributions are light-tailed. It is useful to rewrite $R_F(p)$ in terms of the aforementioned CTE/TVaR and PHT risk measures, which is accomplished by the equation
\[
\mathbf{E}[X|X \le Q(p)]=\bigg ( 1- {1\over p} \bigg )\mathbf{E}[X|X > Q(p)]+{1\over p}\mathbf{E}[X].
\]
Hence,
\[
R_F(p)=\mathcal{H}_p \big ( \mathbf{E}[X|X > Q(p)], \mathbf{E}[X] \big ),
\]
where the coupling function $\mathcal{H}_p:[0,\infty ) \times [0,\infty ) \to [0,\infty )$ is given by the formula
\[
\mathcal{H}_p(x,y)=1- {1\over p} +{1\over p}\bigg ( { y \over x } \bigg ).
\]
This concludes Example \ref{ex-zenga}.
\end{example}

Hence, the class of interesting and useful coupling functions spans beyond ratios $x/y$. This suggests considering a most general `coupled risk measure,' defined as follows: Let $\mathcal{H}:[0,\infty ) \times [0,\infty ) \to [0,\infty )$ be a `coupling' function, which couples two basic risk measures
\[
L_{i}[F]=\int_{0}^{1} H_{i}\circ Q(s)d\Psi_{i}(s), \quad i=1,2,
\]
into one
\[
\Pi[F]=\mathcal{H}(L_{1}[F],L_{2}[F]),
\]
which we call the coupled risk measure. Our task in this paper is to develop a statistical inferential theory for this risk measure when loss variables follow heavy-tailed distributions. The corresponding theory in the case of light-tailed distributions is already available in the literature, and we refer to Brazauskas et al. (2007, 2009) and references therein for details.

The rest of this paper is organized as follows. In Section \ref{sect-2} we introduce an empirical estimator, denoted by $\widehat{\Pi}_{n}$, of the coupled risk measure $\Pi[F]$ when losses are heavy-tailed. In Section \ref{sect-3a} we establish weak approximations and thus asymptotic normality of the estimator $\widehat{\Pi}_{n}$ under several sets of conditions. In Section \ref{sect-3b} we give illustrative examples of the aforementioned weak approximations. Proofs are given in Section \ref{sect-4}.

\section{\bf Constructing an estimator for $\Pi[F]$}
\label{sect-2}

The cdf $F$ is unknown, and thus the risk measure $\Pi[F]$ is unknown. Estimating $\Pi\lbrack F]$ crucially relies on estimating $L_{i}[F]$. To work out our initial intuition on the topic, we start with a brief discussion of what happens within the classical CLT-like framework, and why we need to depart from it.

It is natural to construct an estimator for $L_{i}[F]$ by simply replacing the (unknown) population cdf $F$ by its empirical counterpart $F_{n}$, which gives the weighted sum
\[
\widetilde{L}_{i,n}=\sum_{j=1}^{n}c_{i,j,n}  H_{i}(X_{j:n})
\]
of the (observable) random variables $H_{i}(X_{j:n})$, $1\le j \le n$,
with the coefficients
\[
c_{i,j,n}=\Psi_{i}\left(  \frac{j}{n}\right)-\Psi_{i}\left(  \frac{j-1}{n}\right) .
\]
It is known from the theory of $L$-statistics (e.g., Shorack and Wellner, 1986) that under some assumptions on $H_{i}$, $\Psi_{i}$, and $Q$, the following asymptotic-normality result holds:
\begin{equation}
\sqrt{n}\,(\widetilde{L}_{i,n}-L_{i}[F])\rightarrow_{d}\mathcal{N}(0,\sigma_{F}^{2}),
\label{L-stats0}
\end{equation}
provided that the variance
\[
\sigma_{F}^{2}=\int_{0}^{1}\int_{0}^{1}(H_{i}\circ Q)'(s)(H_{i}\circ
Q)'(t)(\min(s,t)-st)d\Psi_{i}(s)d\Psi_{i}(t)
\]
is finite. In the case of the CTE/TVaR risk measure, this scenario has been thoroughly investigated by Brazauskas et al. (2008), and for general $L$-type risk measures by Brazauskas et al. (2007, 2009). The finiteness of the variance $\sigma_{F}^{2}$ is, however, often violated by heavy-tailed distributions. Hence, we need to develop another approach for deriving statistical inferential results in the case of such distributions, and we shall do so next.

We note at the outset that a number of special cases that are covered by the coupled risk measure $\Pi[F]$ have been investigated in the literature within the heavy-tailed framework. For example, Peng (2001) has established a ground-breaking statistical inferential theory for the net premium. Necir et al. (2007), Necir and Meraghni (2009) have developed an analogous theory for the proportional hazards transform. Necir et al. (2010) have tackled the conditional tail expectation. Necir and Meraghni (2010) have devoted their research to general $L$-functionals. All of these risk measures are special cases of the above introduced coupled risk measure $\Pi[F]$, and thus our inferential theory developed in the following sections will cover all these special cases.

Thus, our task now is to modify the classical estimator $\widetilde{L}_{i,n}$ in such a way that it would work in the heavy-tailed setting. For this, keeping in mind that high quantiles are estimated differently from the intermediate ones of the (observable) random variables $H_{i}(X_{1}),\dots,H_{i}(X_{n})$, we introduce integers $k=k_{n}$, which depend on $n$ and are such that
\begin{equation}
k\rightarrow\infty \quad \textrm{and} \quad k/n\rightarrow0,
\label{con-00e}
\end{equation}
with further assumptions specified later in this paper. Next we write the decomposition $L_{i}[F]=L_{i,n}(1)+L_{i,n}(2)$ with
\[
L_{i,n}(1)=\int_{0}^{1-k/n}H_{i}\circ Q(s)d\Psi_{i}(s)
\]
and
\[
L_{i,n}(2)=\int_{1-k/n}^{1}H_{i}\circ Q(s)d\Psi_{i}(s).
\]
Finally, we replace the quantile function $Q$ in the definition of $L_{i,n}(1)$ by the
classical non-parametric estimator $Q_{n}$, and then replace $H_{i}\circ Q$ in the definition of $L_{i,n}(2)$ by any of the many high-quantile estimators available in the literature (see, e.g., Beirlant et al., 2004; Castillo et al. 2005; and references therein), which we denote by $\widehat{H_{i}\circ Q}$. Throughout the present paper we work with the Weissman (1978) estimator
\[
\widehat{H_{i}\circ Q}(s)=\left(  \frac{k}{n}\right)  ^{\widehat{\gamma}_{i}}H_{i}(X_{n-k:n})(1-s)^{-\widehat{\gamma}_{i}},\quad s\in\big (1-k/n,1\big ),
\]
where $\widehat{\gamma}_{i}$ is the Hill (1975) estimator of the tail index $\gamma_{i}\in(1/2,1)$ defined by the formula
\[
\widehat{\gamma}_{i}={\frac{1}{k}}\sum_{j=1}^{k}
\log \bigg ( {H_{i}(X_{n-j+1:n}) \over H_{i}(X_{n-k:n})} \bigg ).
\]
In summary, we have arrived at the estimator
\begin{equation}
\widehat{L}_{i,n}=\widehat{L}_{i,n}(1)+\widehat{L}_{i,n}(2),
\label{decomp-emp}
\end{equation}
where
\[
\widehat{L}_{i,n}(1)   =\sum_{j=1}^{n-k}c_{i,j,n}  H_{i}\left(X_{j:n}\right)
\]
and
\[
\widehat{L}_{i,n}(2) = c_{i,\vartriangle,n} H_{i}(X_{n-k:n})
\]
with
\[
c_{i,\vartriangle,n}=\left(  \frac{k}{n}\right)
^{\widehat{\gamma}_{i}}\int_{1-k/n}^{1}(1-s)^{-\widehat{\gamma}_{i}}d\Psi
_{i}(s).
\]
Note the similarity between $\widehat{L}_{i,n}(1)$ and the `classical' estimator $\widetilde{L}_{i,n}$, but $\widehat{L}_{i,n}(2)$ and $\widetilde{L}_{i,n}$ are quite different. Replacing $L_{1}[F]$ and $L_{2}[F]$ in $\mathcal{H}(L_{1}[F],L_{2}[F])$ by the above constructed $\widehat{L}_{1,n}$ and $\widehat{L}_{2,n}$, we obtain the estimator
\[
\widehat{\Pi}_{n}=\mathcal{H}(\widehat{L}_{1,n},\widehat{L}_{2,n})
\]
of $\Pi[F]$. In the next section we shall establish the asymptotic distribution of $\widehat{\Pi}_{n}$.

For developing statistical inferential results (e.g., confidence intervals and hypothesis tests) for $\Pi[F]$, we need to derive asymptotic distributions of $\widehat{L}_{1,n}$ and $\widehat{L}_{2,n}$, which we do in the next section.
The following regular variation condition plays a decisive role in the derivations. Namely, let the function $H_{i}\circ Q$ be regularly varying at $1$ with index $-\gamma_{i}<0$, that is,
\begin{equation}
\lim_{\epsilon\downarrow0}{\frac{H_{i}\circ Q(1-\epsilon s)}{H_{i}\circ
Q(1-\epsilon)}}=s^{-\gamma_{i}}
\label{tag-2u}
\end{equation}
for every $s>0$. This is the so-called first-order condition, which is sufficient  for proving consistency of the estimator $\widehat{\Pi}_{n}$. For establishing the asymptotic distribution of the estimator, we need a second-order condition, which specifies the rate of convergence in the first-order condition. Namely, assume that, for every $s>0$,
\begin{equation}
\lim_{\epsilon\downarrow0}{\frac{1}{A_{i}(1/\epsilon)}}\left( {\frac {H_{i}\circ Q(1-\epsilon s)}{H_{i}\circ Q(1-\epsilon)}}-s^{-\gamma_{i}}\right)  =s^{-\gamma_{i}}\frac{s^{-\omega_{i}}-1}{\omega_{i}} ,
\label{tag-2A}
\end{equation}
where $\omega_{i}\leq0$ is the so-called second-order parameter and
$x\mapsto A_{i}(x)$ is a function that does not change its
sign for all sufficiently large $x$ and converges to $0$ when $x\uparrow \infty $.
When $\omega_{i}=0$, then the ratio on the right-hand side of equation
(\ref{tag-2A}) is interpreted as $\log s$.

\section{\bf Establishing the asymptotic distribution of $\widehat{\Pi}_{n}$}
\label{sect-3a}

We assume that the first partial derivatives of the functional $\mathcal{H}$ are continuous at the point $(L_{1}[F],L_{2}[F])$, and not both of them are zero at the point. By $\mathcal{H}^{(1)}_{x}(L_{1}[F],L_{2}[F])$ and $\mathcal{H}^{(1)}_{y}(L_{1}[F],L_{2}[F])$ we denote the first partial
derivatives  of $\mathcal{H}(x,y)$ with respect to $x$ and $y$, respectively, both evaluated at the point $(x,y)=(L_{1}[F],L_{2}[F])$. The following lemma plays a fundamental role in our considerations.

\begin{lemma}\label{lemma-funda}
Assume that there are standard Brownian bridges $B_{n}$ defined on a possibly different probability space than the original one, and also assume that, for each $i=1$ and $i=2$, there are constants $b_{i,n}$ and $\lambda_{i}$, and linear functionals $\ell_{i,n}$ defined on the space of all functions such that
\begin{equation}
b_{i,n}\big(\widehat{L}_{i,n}-L_{i}[F]\big)
=\lambda_{i}+\ell_{i,n}(B_{n})+o_{\mathbf{P}}(1)
\label{approxim-1}
\end{equation}
when $n\rightarrow\infty$. Furthermore, let the constants $b_{i,n}$ be such that
\begin{equation}
{\frac{b_{2,n}}{b_{1,n}+b_{2,n}}}\rightarrow\delta
\label{approxim-2}
\end{equation}
for some $\delta\in [0,1]$. Then, when $n\rightarrow\infty$,
\begin{equation}
{\frac{b_{1,n}b_{2,n}}{b_{1,n}+b_{2,n}}}\big(\widehat{\Pi}_{n}-\Pi[F]\big)
=\lambda+\ell_n(B_{n})+o_{\mathbf{P}}(1)
\label{approxim-1a}
\end{equation}
with the bias term
\[
\lambda=\delta\mathcal{H}^{(1)}_{x}(L_{1}[F],L_{2}[F])\lambda
_{1}+(1-\delta)\mathcal{H}^{(1)}_{y}(L_{1}[F],L_{2}[F])\lambda_{2}%
\]
and the linear functional
\[
\ell_n=\delta\mathcal{H}^{(1)}_{x}(L_{1}[F],L_{2}[F])\ell_{1,n}
+(1-\delta)\mathcal{H}^{(1)}_{y}(L_{1}[F],L_{2}[F])\ell_{2,n}.
\]
\end{lemma}

Lemma \ref{lemma-funda} reduces our task of deriving the asymptotic distribution of the coupled risk measure $\Pi[F]$ to establishing statement (\ref{approxim-1}) for both $i=1$ and $i=2$. Recall that $L_{i}=L_{i,n}(1)+L_{i,n}(2)$ and $\widehat{L}_{i,n}=\widehat{L}_{i,n}(1)+\widehat{L}_{i,n}(2)$. These two equations further reduce our task to showing that, with the above defined constants $b_{i,n}$ and Brownian bridges $B_{n}$, for each $\Delta =1$ and $\Delta =2$ there are constants $\lambda_{i}(\Delta)$ and linear functionals $\ell_{i,n}(\Delta;\bullet)$ such that, when $n\rightarrow\infty$,
\begin{equation}
b_{i,n}\big(\widehat{L}_{i,n}(\Delta)-L_{i,n}(\Delta)\big)
=\lambda_{i}(\Delta)+\ell_{i,n}(\Delta;B_{n})+o_{\mathbf{P}}(1).
\label{approxim-4}
\end{equation}
In view of this assumption, the earlier noted constants $\lambda_{i}$ and linear functionals $\ell_{i}$ are the sums
\begin{equation}
\lambda_{i}=\lambda_{i}(1)+\lambda_{i}(2)
\label{approxim-4con}
\end{equation}
and
\begin{equation}
\ell_{i,n}=\ell_{i,n}(1;\bullet)+\ell_{i,n}(2;\bullet).
\label{approxim-4lin}
\end{equation}
In various generalities, statement (\ref{approxim-4}) when $\Delta =1$ will be established in Theorems \ref{theor-1ab} and \ref{theor-1a} below, whereas Theorem \ref{theor-1b} and Corollary \ref{utheor-1b} will deal with the case $\Delta=2$.

\begin{note}\label{probab-space}\rm
In all of the following results, the Brownian bridges $B_{n}$ and the probability space on which they are defined can be chosen same, and this is a crucial property as it allows us to combine various asymptotic results into one.
\end{note}

In what follows we use the notation
\[
D_{i,n}=H_{i}\circ Q(1-k/n)\bigg({k \over n}\bigg )^{\gamma_{i}}\int_{1-k/n}^{1}(1-s)^{-\gamma
_{i}}d\Psi_{i}(s)
\]
and assume that, when $n\rightarrow\infty$,
\begin{equation}
\sqrt{{n\over k}} ~ D_{i,n}\rightarrow\infty .
\label{con-0e}
\end{equation}

The next theorem establishes statement (\ref{approxim-4}) when $\Delta =1$ under a very weak assumption on the function $F\circ H_{i}^{-1}$, assuming only that it is continuous, but this generality is achieved at the expense of requiring $\Psi_{i}(s)=s$ for all $s\in [0,1]$, which is indeed a restriction though luckily satisfied by a number of important examples.

\begin{theorem}\label{theor-1ab}
Let the function $F\circ H_{i}^{-1}$ be continuous and $\Psi_{i}(s)=s$ for all $s\in [0,1]$. Then there are Brownian bridges $B_{n}$ (on a possibly different probability space) such that
\begin{equation}
\frac{\sqrt{n}(\widehat{L}_{i,n}(1)-L_{i,n}(1)}{(n/k)^{1/2}D_{i,n}}=\ell
_{i,n}(1;B_{n})+o_{\mathbf{P}}(1)
\label{equa-10cc1}
\end{equation}
with
\[
\ell_{i,n}(1;B_{n})={\frac{-1}{(n/k)^{1/2}D_{i,n}}}\int_{0}^{1-k/n}%
B_{n}(s)dH_{i}\circ Q(s).
\]
\end{theorem}

The following theorem allows for a general class of functions $\Psi_{i}$ but imposes a requirement that $H_{i}\circ Q$ is continuously differentiable.

\begin{theorem}\label{theor-1a}
Let $H_{i}\circ Q$ be continuously differentiable on $[0,1)$. Furthermore, let $\sqrt{n}\,(\Psi_{i}(1/n)-\Psi_{i}(0))=O(1)$ when $n\rightarrow\infty$, and let $s\rightarrow\int_{s}^{1}(1-t)^{-\gamma_{i}}d\Psi
_{i}(t)$ be regularly varying at $1$ with an index $\kappa\in(0,1/2)$.  Then there
are Brownian bridges $B_{n}$ (on a possibly different probability space) such that statement (\ref{equa-10cc1}) holds with
\[
\ell_{i,n}(1;B_{n})={\frac{-1}{(n/k)^{1/2}D_{i,n}}}\int_{0}^{1-k/n}(H_{i}\circ
Q)'(s)B_{n}(s)d\Psi_{i}(s).
\]
\end{theorem}

The previous two theorems deal with low and moderate quantiles. The next theorem and its corollary deal with high quantiles.

\begin{theorem}\label{theor-1b}
Let the following three limits exist and be finite:
\begin{equation}
b_{i}=\lim_{n\rightarrow\infty}\sqrt{k}\,A_{i}(n/k),
\label{con-0a}
\end{equation}
\begin{equation}
c_{i}=\lim_{n\rightarrow\infty}{\frac{\int_{0}^{1}s^{-\gamma_{i}}\log
(s)d\Psi_{i}(1-ks/n)}{\int_{0}^{1}s^{-\gamma_{i}}d\Psi_{i}(1-ks/n)}},
\label{con-0b}
\end{equation}
\begin{equation}
d_{i}=\lim_{n\rightarrow\infty}\frac{\int_{0}^{1}s^{-\gamma_{i}}(s^{-\omega_{i}%
}-1)d\Psi_{i}\left(  1-ks/n\right)  }{\int_{0}^{1}s^{-\gamma_{i}}d\Psi
_{i}\left(  1-ks/n\right)  }.
\label{con-0c}
\end{equation}
Furthermore, assume that, for some $\delta>0$ and when $n\rightarrow\infty$,
\begin{equation}
{\frac{\int_{0}^{1}s^{-\gamma_{i}-\omega_{i}-\delta}d\Psi_{i}(1-ks/n)}{\int
_{0}^{1}s^{-\gamma_{i}}d\Psi_{i}(1-ks/n)}}=O(1).
\label{equa-10c}
\end{equation}
Then there are Brownian bridges $B_{n}$ (on a possibly different probability space) such that
\begin{equation}
\frac{\sqrt{n}(\widehat{L}_{i,n}(2)-L_{i,n}(2))}{(n/k)^{1/2}D_{i,n}}
={\frac{-b_{i}d_{i}}{\omega_{i}}}+\ell_{i,n}(2;B_{n})+o_{\mathbf{P}}(1),
\label{equa-10cc2}
\end{equation}
where
\[
\ell_{i,n}(2;B_{n})=-\gamma_{i}(1+c_{i})\sqrt{{\frac{n}{k}}}\,B_{n}(1-k/n)
+c_{i}\gamma_{i}\sqrt{{\frac{n}{k}}}\,\int_{1-k/n}^{1}\frac{B_{n}(s)}{1-s}ds.
\]
\end{theorem}

The following corollary is a special case of Theorem \ref{theor-1b} when $\Psi_{i}(s)=s$.

\begin{corollary}\label{utheor-1b}
Let $\Psi_{i}(s)=s$ for all $s\in [0,1]$, and let the limit $b_i$ defined in (\ref{con-0a}) exist and be finite. Furthermore, let $\gamma_{i}+\omega_{i}<1$. Then there are Brownian bridges $B_{n}$ (on a possibly different probability space) such that
\begin{equation}
\frac{\sqrt{n}(\widehat{L}_{i,n}(2)-L_{i,n}(2))}{(n/k)^{1/2}D_{i,n}}%
={\frac{-b_{i}d_{i}}{\omega_{i}}}+\ell_{i,n}(2;B_{n})+o_{\mathbf{P}}(1),
\label{uequa-10cc2}
\end{equation}
where $c_{i}$ in the definition of $\ell_{i,n}(2;B_{n})$ and the constant $d_{i}$ are given by the formulas
\[
c_{i}={-1 \over 1-\gamma_{i}}
\quad \textrm{and} \quad
d_{i}={\omega_{i} \over 1-\gamma_{i}-\omega_{i}}.
\]
\end{corollary}

We have noted above that the Brownian bridges $B_{n}$ and the probability space are same in all our results. This allows us to pair, for example, Theorem \ref{theor-1ab} with Corollary \ref{theor-1b}, and Theorem \ref{theor-1a} with Theorem \ref{theor-1b}. Two results of this pairing are given next.

\begin{corollary}\label{cor-1}
Assume that one of the following two requirements is satisfied:
\begin{enumerate}
  \item $\Psi_i(x)=s$ for all $s\in [0,1]$, and let the conditions of Theorem \ref{theor-1ab} and Corollary \ref{theor-1b} be satisfied.
  \item The conditions of Theorems \ref{theor-1a} and \ref{theor-1b} are satisfied.
\end{enumerate}
Then there are Brownian bridges $B_{n}$ (on a possibly different probability space) such that
\begin{equation}
\frac{\sqrt{n}(\widehat{L}_{i,n}-L_{i})}{(n/k)^{1/2}D_{i,n}}={\frac
{-b_{i}d_{i}}{\omega_{i}}}+\ell_{i,n}(B_{n})+o_{\mathbf{P}}(1), \label{equa-10fin}
\end{equation}
where $\ell_{i,n}=\ell_{i,n}(1;\bullet )+\ell_{i,n}(2;\bullet)$.
\end{corollary}

Corollary \ref{cor-1} together with Lemma \ref{lemma-funda} imply the following theorem, which we view as the main result of the present paper.

\begin{theorem}\label{th-main-2}
Assume that when $i=1$, and also when $i=2$, one of the following two requirements is satisfied:
\begin{enumerate}
  \item $\Psi_i(x)=s$ for all $s\in [0,1]$, and let the conditions of Theorem \ref{theor-1ab} and Corollary \ref{theor-1b} be satisfied.
  \item The conditions of Theorems \ref{theor-1a} and \ref{theor-1b} are satisfied.
\end{enumerate}
Furthermore, let there exist a constant $\delta\in [0,1]$ such that, when $n\rightarrow\infty$,
\begin{equation}
{\frac{D_{1,n}}{D_{1,n}+D_{2,n}}}\rightarrow\delta.
\label{approxim-5a}
\end{equation}
Then there are Brownian bridges $B_{n}$ (on a possibly different probability space) such that
\begin{equation}
\frac{\sqrt{n}\big(\widehat{\Pi}_{n}-\Pi[F]\big)}{(n/k)^{1/2}(D_{1,n}+D_{2,n})}
=\lambda+\ell_n(B_{n})+o_{\mathbf{P}}(1)
\label{approxim-5b}
\end{equation}
with the bias term
\[
\lambda=\delta\mathcal{H}^{(1)}_{x}(L_{1}[F],L_{2}[F]){\frac{-b_{1}d_{1}}{\omega_{1}}}
+(1-\delta)\mathcal{H}^{(1)}_{y}(L_{1}[F],L_{2}[F]){\frac{-b_{2}d_{2}}{\omega_{2}}}
\]
and the leading Gaussian term
\[
\ell_n=\delta\mathcal{H}^{(1)}_{x}(L_{1}[F],L_{2}[F])\ell_{1,n}
+(1-\delta)\mathcal{H}^{(1)}_{y}(L_{1}[F],L_{2}[F])\ell_{2,n},
\]
where $\ell_{i,n}=\ell_{i,n}(1;\bullet )+\ell_{i,n}(2;\bullet)$.
\end{theorem}

We can use Theorem \ref{th-main-2} to derive asymptotic distributions of various special cases of $\widehat{\Pi}_{n}$ and thus obtain statistical inferential results for the corresponding $\Pi[F]$. Naturally, one of the most challenging tasks that inevitably arises is obtaining explicit expressions for $\ell_n(B_{n})$. To illustrate how this can be done, in the next section we present detailed examples tackling the PHT and CTE/TVaR risk measures.

\section{\bf Brownian bridge approximations: illustrative examples}
\label{sect-3b}

Our illustrative examples in Section \ref{sect-1} were based on three risk measures: the proportional hazards transform $\textrm{PHT}[\rho,F]$, the conditional tail expectation $\textrm{CTE}[t,F]$, and the net premium $\mathbf{E}[X]$, which can be viewed, for example, as the proportional hazards transform $\textrm{PHT}[\rho,F]$ with the parameter $\rho =1$. Hence, in order to develop a statistical inferential theory for $\Pi[F]=\mathcal{H}(L_{1}[F],L_{2}[F])$ with $L_{1}[F]$ and $L_{2}[F]$ being two of the aforementioned three risk measures, we need to establish asymptotic representations (\ref{equa-10fin}) for the proportional hazards transform (PHT) and the conditional tail expectation (CTE). This we do next. Since we work individually with each of the two risk measures, we shall drop the subindex $i$ from the rest of this section.

\subsection*{PHT}

In the case of the proportional hazards transform $\textrm{PHT}[\rho,F]$, we have   $\Psi(s)=-(1-s)^{1/\rho}$ and $H(x)=x$ (Example \ref{ex-1.1}). Using decomposition (\ref{decomp-emp}), the estimator $\widehat{PHT}_n$ can be expressed as
\[
\widehat{PHT}_n= \sum_{j=1}^{n-k}
\bigg\{ \bigg(  1-\frac{j-1}{n}\bigg)  ^{1/\rho}-\bigg(  1-\frac{j}{n}\bigg)  ^{1/\rho}\bigg \}X_{j:n}
+\bigg( \frac{k}{n}\bigg)^{1/\rho}\frac{X_{n-k:n}}{1-\rho\widehat{\gamma}}
\]
with
\begin{equation}
\widehat{\gamma}={1\over k}\sum_{j=1}^{k}
\log\bigg ( {X_{n-j+1:n} \over X_{n-k:n}} \bigg ).
\label{hill-1}
\end{equation}
Brownian bridge approximations for $\widehat{PHT}_n$ can be established using Corollary \ref{cor-1} under the conditions of Theorems \ref{theor-1a} and \ref{theor-1b}. For this, we first calculate
\[
D_{n}=\bigg( \frac{k}{n}\bigg)^{1/\rho}\frac{Q(1-k/n)}{1-\rho\gamma}
\]
and easily check that $(n/k)^{1/2}D_{n}\to \infty $ when $n\to \infty $.  Assuming that the quantile function $Q$ is continuously differentiable on $[0,1)$,  with other conditions of Theorem \ref{theor-1a} satisfied automatically, we have that
\begin{equation}
\ell_{n}(1;B_{n})={\frac{-(1/\rho-\gamma)}{(k/n)^{1/\rho -1/2}Q(1-k/n)}}\int
_{0}^{1-k/n}(1-s)^{1/\rho-1}B_{n}(s)dQ(s).
\label{ell-pht-1}
\end{equation}
Assume that $\lim_{n\rightarrow\infty}\sqrt{k}\,A_{1}(n/k)=0$. Limits (\ref{con-0b}) and (\ref{con-0c}) are
\begin{equation}
c={-1 \over 1/\rho-\gamma}
\quad \textrm{and} \quad
d={\omega \over 1/\rho-\gamma-\omega},
\label{lim-cANDd}
\end{equation}
provided that $1/\rho-\gamma-\omega>0$. The ratio on the right hand side of equation (\ref{equa-10c}) is equal to $( 1/\rho -\gamma )/( 1/\rho -\gamma-\omega-\delta) $ under the assumption that $\delta>0$ is so small that $\delta<1/\rho-\gamma-\omega$. (We can always choose such $\delta>0$ because $1/\rho-\gamma-\omega>0$ by assumption.)  We can now conclude from Theorem \ref{theor-1b} that
\begin{equation}
\ell_{n}(2;B_{n})=-\gamma \bigg ( 1- {1 \over 1/\rho-\gamma}\bigg ) \sqrt{{\frac{n}{k}}}\,B_{n}(1-k/n)
-\gamma{ 1 \over 1/\rho-\gamma}\sqrt{{\frac{n}{k}}}\,\int_{1-k/n}^{1}\frac{B_{n}(s)}{1-s}ds.
\label{ell-pht-2}
\end{equation}
Hence, $\ell_{n}(B_{n})$ is the sum of the right-hand sides of equations (\ref{ell-pht-1}) and (\ref{ell-pht-2}). The sum is asymptotically Gaussian with the mean zero and the variance (cf.\, Necir and Meraghni, 2009)
\begin{equation}
\frac{\gamma^{2}(  \gamma^{2}\rho^{2}-2\gamma^{2}\rho^{3}+4\gamma\rho^{2}-2\gamma\rho+\rho^{2}-2\rho+1)  }{(\gamma\rho-1)^{2}}+\frac{2\gamma^{2}(\rho+\gamma\rho-1)  }{\rho+2\gamma\rho-2}.
\label{var-1}
\end{equation}

\subsection*{An interlude}

We note that it is not the variance expression (\ref{var-1}) that is important for us --  unless we consider the proportional hazards transform $\textrm{PHT}[\rho,F]$ as a stand alone risk measure -- but the above derived expression for $\ell_{n}(B_{n})$ as it allows us to combine the expression with the corresponding one of the other risk measure making up the coupled risk measure under consideration. This explains our interest in deriving expressions of $\ell_{n}(B_{n})$ for various risk measures. For example, if the second risk measure is the mean, which is the net premium, then an expression for the corresponding $\ell_{n}(B_{n})$ follows from the formulas of the previous subsection by setting $\rho =1$. If the second risk measure is the CTE/TVaR, then we refer to our next subsection.

\subsection*{CTE/TVaR}

In the case of the conditional tail expectation $\textrm{CTE}[t,F]$, we have   $\Psi(s)=(s-t)_{+} / (1-t)$ and $H(x)=x$ (Example \ref{ex-1.1}). Using decomposition (\ref{decomp-emp}), we express the estimator $\widehat{CTE}_n$ by the formula
\[
\widehat{CTE}_n=\dfrac{1}{1-t}\sum_{j=1}^{n-k} \bigg ( \left(  \frac{j}{n}-t\right)_{+}-\left(  \frac{j-1}{n}-t\right)_{+}  \bigg )  X_{j:n}+\frac{kX_{n-k:n}}{n\left(  1-t\right)  \left(  1-\widehat{\gamma}\right)  }
\]
for any fixed $0<t<1$ and for all sufficiently large $n$, where $\widehat{\gamma}$ is same as in (\ref{hill-1}). Assume that $\lim_{n\rightarrow\infty}\sqrt{k}\,A_{1}(n/k)=0$. Limits (\ref{con-0b}) and (\ref{con-0c}) are
\[
c={-1 \over 1-\gamma}
\quad \textrm{and} \quad
d={\omega \over 1-\gamma-\omega},
\]
provided that $\gamma+\omega<1$. (Same as in (\ref{lim-cANDd}) but with $\rho=1$.) The ratio on the right hand side of equation (\ref{equa-10c}) is equal to $1/( 1 -\gamma-\omega-\delta) $, assuming that $\delta>0$ is so small that $\delta<1-\gamma-\omega$. From Theorems \ref{theor-1a} and \ref{theor-1b} we have Brownian bridges $B_{n}$ such that
\begin{equation}
\ell_{n}(1;B_{n})=\frac{-(1-\gamma )}{(k/n)^{1/2}Q(1-k/n)} \int_{0}^{1-k/n}B_{n}(s)dQ(s)
\label{stat-1}
\end{equation}
and
\begin{equation}
\ell_{n}(2;B_{n})={\gamma^2 \over 1-\gamma }
 \sqrt{{n\over k}}\,B_{n}(1-k/n)
-{\gamma \over 1-\gamma }
\sqrt{{n\over k}}\,\int_{1-k/n}^{1}\frac{B_{n}(s)}{1-s}ds.
\label{stat-2}
\end{equation}
Hence, $\ell_{n}(B_{n})$ is equal to the sum of the right-hand sides of equations (\ref{stat-1}) and (\ref{stat-2}). This sum is asymptotically Gaussian with the mean zero and the variance (cf.\, Necir et al., 2010)
\begin{equation}
{\frac{\gamma^{4}}{(1-\gamma)^{2}(2\gamma-1)}}.
\label{var-2}
\end{equation}
We highlight again that it is not the latter variance but the above derived formula for $\ell_{n}(B_{n})$ that is of our primary interest in the present paper.

\subsection*{Useful formulas}

Calculating the second moments or, equivalently, the variances of the functionals  $\ell_{n}(B_{n})$ in our above examples reduces to calculating the second and mixed moments of the following three random variables:
\begin{align*}
W_{1,n}&={1\over (k/n)^{1/\rho -1/2}Q(1-k/n)}\int_{0}^{1-k/n}(1-s)^{1/\rho-1}B_{n}(s)dQ(s),
\\
W_{2,n}&= {\sqrt{{n\over k}}\,B_{n}(1-k/n)},
\\
W_{3,n}&= {\sqrt{{n\over k}}\,\int_{1-k/n}^{1}\frac{B_{n}(s)}{1-s}ds}.
\end{align*}
We next give formulas that make such calculations straightforward. Namely, when $n\to \infty $, we have the limits:
\begin{align*}
& \mathbf{E}[W_{1,n}^{2}]\rightarrow {\gamma^2 \over (1/\rho-\gamma-1)(1/\rho-\gamma-1/2)},
\\
& \mathbf{E}[W_{2,n}^{2}]\rightarrow 1,
\\
& \mathbf{E}[W_{3,n}^{2}]\rightarrow 2,
\\
& \mathbf{E}[W_{1,n}W_{2,n}]\rightarrow {\gamma \rho \over 1/\rho-\gamma-1},
\\
& \mathbf{E}[W_{1,n}W_{3,n}]\rightarrow {-\gamma  \over 1/\rho-\gamma-1},
\\
& \mathbf{E}[W_{2,n}W_{3,n}]\rightarrow 1.
\end{align*}
Using these limits, we can now easily check, for example, that $\mathbf{E}[\ell_{n}^2(B_{n})]$ is asymptotically (when $n\to \infty $) equal to quantity (\ref{var-1}) in the PHT case and to (\ref{var-2}) in the CTE/TVaR case.

\section{\bf Proofs}
\label{sect-4}

We begin with a proof of Theorem \ref{theor-1a}, which is then followed by a proof of Theorem \ref{theor-1ab}. This sequence of proofs is natural since Theorem \ref{theor-1ab} assumes $\Psi_{i}(s)=s$ and thus allows us to employ an additional technical tool, called the Vervaat process, which enables us to relax the differentiability of $F\circ H_{i}^{-1}$ to only continuity.

\begin{proof}[\textbf{Proof of Theorem \ref{theor-1a}}]

We first show that
\begin{equation}
{\frac{\sqrt{n}}{(n/k)^{1/2}D_{i,n}}}\int_{0}^{1/n}\left(  H_{i}\circ
Q_{n}(s)-H_{i}\circ Q(s)\right)  d\Psi_{i}(s)=o_{\mathbf{P}}(1).
\label{stat-1a}
\end{equation}
Since $Q_{n}(s)=X_{1:n}$ when $s\in (0,1/n]$ and
$X_{1:n}=O_{\mathbf{p}}(1)$, and since the function $H_{i}$ is non-decreasing,
we have that $\sup_{s\in (0,1/n]}H_{i}\circ Q_{n}(s)=O_{\mathbf{p}}(1) $. Consequently,  the quantity $\sqrt{n}\int_{0}^{1/n}H_{i}\circ Q_{n}(s)ds$ is of the order
$O_{\mathbf{p}}( 1/\sqrt{n})  $ and thus, in particular, converges to zero in probability. Furthermore, since the function $H_{i}\circ Q$ is non-decreasing, we have that
\begin{equation}
\sqrt{n}\int_{0}^{1/n}H_{i}\circ Q(s)d\Psi_{i}(s)
\leq \sqrt{n}(\Psi_{i}(1/n) -\Psi_{i}(0) ) H_{i}\circ Q(1/n) .
\label{stat-1ak}
\end{equation}
By assumption, $\sqrt{n}(\Psi_{i}(1/n) -\Psi_{i}(0) )$ is bounded, and so the left-hand side of inequality (\ref{stat-1ak}) is asymptotically bounded. This proves statement (\ref{stat-1a}) because $(n/k)^{1/2}D_{i,n}\rightarrow \infty $ when $n\rightarrow\infty$ by assumption (\ref{con-0e}).

In view of statement (\ref{stat-1a}), the proof of Theorem \ref{theor-1a} reduces to showing that
\begin{equation}
{\frac{\sqrt{n}}{(n/k)^{1/2}D_{i,n}}}\int_{1/n}^{1-k/n}\left(  H_{i}\circ
Q_{n}(s)-H_{i}\circ Q(s)\right)  d\Psi_{i}(s)=\ell_{i,n}(1;B_{n})+o_{\mathbf{P}}(1). \label{stat-1new}
\end{equation}
Before doing so, we introduce additional notation. Let $E_{n}^{-1}$ denote the uniform quantile function based on the uniform on $[0,1]$ random variables $F(X_{1}),\dots,F(X_{n})$ (recall that the cdf $F$ is continuous by assumption).
Hence, $Q_{n}(s)=Q(E_{n}^{-1}(s))$. With the notation
\begin{equation}
\vartheta_{n}(s)=E_{n}^{-1}(s)-s ,
\label{quant-process}
\end{equation}
and using the assumed differentiability of the function $H_{i}\circ Q$, we have that
\begin{align}
& {\frac{\sqrt{n}}{(n/k)^{1/2}D_{i,n}}} \int_{1/n}^{1-k/n}\left(  H_{i}\circ
Q(s+\vartheta_{n}(s))-H_{i}\circ Q(s)\right)  d\Psi_{i}(s)
\notag
\\
&={\frac{\sqrt{n}}{(n/k)^{1/2}D_{i,n}}}\int_{1/n}^{1-k/n}\mathbf{E}_{\tau
}\left[  (H_{i}\circ Q)'(s+\tau\vartheta_{n}(s))\right]  \vartheta
_{n}(s)d\Psi_{i}(s)
\notag
\\
&= {\frac{\sqrt{n}}{(n/k)^{1/2}D_{i,n}}}\int_{1/n}^{1-k/n}\Big (\mathbf{E}_{\tau
}\left[  (H_{i}\circ Q)'(s+\tau\vartheta_{n}(s))\right]
-(H_{i}\circ Q)'(s)\Big )\vartheta_{n}(s)d\Psi_{i}(s)
\notag
\\
&\quad +{\frac{\sqrt{n}}{(n/k)^{1/2}D_{i,n}}}\int_{1/n}^{1-k/n}(H_{i}\circ
Q)'(s)\vartheta_{n}(s)d\Psi_{i}(s)
\label{stat-1b}
\end{align}
where $\tau$ is a uniform on $[0,1]$ random variable, independent of all
other random variables, and $\mathbf{E}_{\tau}$ denotes the expectation with
respect to $\tau$, with all other random variables being fixed.

We shall next prove that
\begin{equation}
{\frac{\sqrt{n}}{(n/k)^{1/2}D_{i,n}}}\int_{1/n}^{1-k/n}\Big (\mathbf{E}_{\tau
}\left[  (H_{i}\circ Q)'(s+\tau\vartheta_{n}(s))\right]
-(H_{i}\circ Q)'(s)\Big )\vartheta_{n}(s)d\Psi_{i}(s)=o_{\mathbf{P}}(1).
\label{stat-1d}
\end{equation}
For this we first fix any $\epsilon\in(0,1)$ and establish statement (\ref{stat-1d}) with the integral $\int_{1/n}^{1-k/n}$ replaced by $\int_{1/n}^{1-\epsilon}$. The function $(H_{i}\circ Q)'$ is uniformly continuous on the interval $[0,1-\epsilon]$ and the process $\vartheta_{n}$ converges to $0$ uniformly on $[0,1]$. Moreover, $\sup_{0<s<1}|\sqrt{n}\,\vartheta_{n}(s)|=O_{\mathbf{P}}(1)$. Since $(n/k)^{1/2}D_{i,n}\rightarrow\infty$, we therefore have that
\begin{equation}
{\frac{\sqrt{n}}{(n/k)^{1/2}D_{i,n}}}\int_{1/n}^{1-\epsilon}|\vartheta
_{n}(s)|d\Psi_{i}(s)=O_{\mathbf{P}}(1).
\label{stat-1d4}
\end{equation}
To complete the proof of statement (\ref{stat-1d}), we are left to
verify statement (\ref{stat-1d}) with the integral $\int_{1/n}^{1-k/n}$ replaced by $\int_{1-\epsilon}^{1-k/n}$, that is, we need to prove that
\begin{equation}
{\frac{\sqrt{n}}{(n/k)^{1/2}D_{i,n}}}\int_{1-\epsilon}^{1-k/n}\Big (\mathbf{E}%
_{\tau}\left[  (H_{i}\circ Q)'(s+\tau\vartheta_{n}(s))\right]
-(H_{i}\circ Q)'(s)\Big )\vartheta_{n}(s)d\Psi_{i}(s)=o_{\mathbf{P}}(1).
\label{stat-1d2}
\end{equation}
To this end, we first use the fact that $\sup_{1-\epsilon<s<1-k/n}{|\sqrt
{n}\,\vartheta_{n}(s)|/(1-s)^{1/2}}=O_{\mathbf{P}}(1)$, which follows from
statement (2.2) on page 40 of Cs\"{o}rg\H{o} et al. (1986) and the proof of statement (2.39) on page 49 of Cs\"{o}rg\H{o} et al. (1986). Then we use the fact that
\begin{equation}
\sup_{1-\epsilon\leq s\leq1-k/n}\bigg |{\frac{\mathbf{E}_{\tau}\left[
(H_{i}\circ Q)'(s+\tau\vartheta_{n}(s))\right]  }{(H_{i}\circ
Q)'(s)}}-1\bigg |=o_{\mathbf{P}}(1), \label{stat-1d6}%
\end{equation}
which is a consequence of Lemma 3 in Necir and Meraghni (2009). Hence, statement (\ref{stat-1d2})
holds if
\begin{equation}
{\frac{1}{(n/k)^{1/2}D_{i,n}}}\int_{1-\epsilon}^{1-k/n}(H_{i}\circ
Q)'(s)^{1/2}d\Psi_{i}(s)=O(1). \label{stat-1d8}%
\end{equation}
To verify statement (\ref{stat-1d8}), we use the fact that
$(1-s)(H_{i}\circ Q)'(s)/H_{i}\circ Q(s)$ is asymptotically bounded when
$s\uparrow1$. This allows us to replace the derivative $(H_{i}\circ
Q)'(s)$ by $(1-s)H_{i}\circ Q(s)$ and reduces the proof of statement (\ref{stat-1d8}) to showing that, when $n\rightarrow \infty $,
\begin{equation}
{\frac{1}{a(1-k/n)b(1-k/n)}}\int_{1-\epsilon}^{1-k/n}a(s)db(s)=O(1),
\label{stat-3b0}
\end{equation}
where $a(s)=H_{i}\circ Q(s)(1-s)^{\gamma_{i}-1/2}$
and $b(s)=\int_{s}^{1}(1-t)^{-\gamma_{i}}d\Psi_{i}(t)$. We shall establish statement (\ref{stat-3b0}) in Lemma \ref{lemma4.1} below, taking now the validity of the statement for granted. This concludes the proof of statement (\ref{stat-1d8}).

In summary, we have prove that
\begin{multline}
{\frac{\sqrt{n}}{(n/k)^{1/2}D_{i,n}}}\int_{1/n}^{1-k/n}\left(  H_{i}\circ
Q_{n}(s)-H_{i}\circ Q(s)\right)  d\Psi_{i}(s)\\
={\frac{\sqrt{n}}{(n/k)^{1/2}D_{i,n}}}\int_{1/n}^{1-k/n}(H_{i}\circ
Q)'(s)\vartheta_{n}(s)d\Psi_{i}(s)+o_{\mathbf{P}}(1).
\label{stat-1c}
\end{multline}
Our next task is to establish a weak approximation for the main term on the right-hand side of equation (\ref{stat-1c}). Namely, we shall show that
\begin{multline}
{\frac{\sqrt{n}}{(n/k)^{1/2}D_{i,n}}}\int_{1/n}^{1-k/n}(H_{i}\circ Q)^{\prime
}(s)\vartheta_{n}(s)d\Psi_{i}(s)\\
={\frac{-1}{(n/k)^{1/2}D_{i,n}}}\int_{1/n}^{1-k/n}(H_{i}\circ Q)^{\prime
}(s)B_{n}(s)d\Psi_{i}(s)+o_{\mathbf{P}}(1),
\label{stat-9e}
\end{multline}
where, defined on a possibly different probability space, $B_{n}$ are Brownian bridges such that (see statement (2.2) on page 40 of Cs\"{o}rg\H{o} et al., 1986) for
any $0\leq\nu<1/2$ we have
\begin{equation}
\sup_{1/n\leq s\leq(n-1)/n}n^{\nu}{\frac{|\sqrt{n}\,\vartheta_{n}%
(s)+B_{n}(s)|}{s^{1/2-\nu}(1-s)^{1/2-\nu}}}=O_{\mathbf{P}}(1)
\label{stat-cchm}
\end{equation}
when $n\rightarrow\infty$.

\begin{note}\rm
The need to change the original probability space into a new one does not affect our results. In fact, we can start working on the probability space of Cs\"{o}rg\H{o} et al. (1986) right at the very outset since all our results are `in probability.'
\end{note}

In view of (\ref{stat-cchm}), statement (\ref{stat-9e}) follows if
\begin{equation}
{\frac{n^{-\nu_{i}}}{(n/k)^{1/2}D_{i,n}}}\int_{1/n}^{1-k/n}(H_{i}\circ
Q)'(s)^{1/2-\nu_{i}}d\Psi_{i}(s)\rightarrow0,
\label{stat-9a}
\end{equation}
where $\nu_{i}\in [0,1/2)$ can be any, even though our arguments below will require
$\nu_{i}$ to be strictly positive, though no matter how small it could be.

To prove statement (\ref{stat-9a}), we employ the earlier noted fact that $(1-s)(H_{i}\circ Q)'(s)/H_{i}\circ Q(s)=O(1)$ when $s\uparrow1$, and reduce the proof to showing that
\begin{equation}
{\frac{1}{k^{\nu_{i}}}}\bigg ({\frac{\int_{1/n}^{1-k/n}H_{i}\circ
Q(s)(1-s)^{\gamma_{i}-1/2-\nu_{i}}(1-s)^{-\gamma_{i}}d\Psi_{i}(s)}{H_{i}\circ
Q(1-k/n)(k/n)^{\gamma_{i}-1/2-\nu_{i}}\int_{1-k/n}^{1}(1-s)^{-\gamma_{i}}%
d\Psi_{i}(s)}}\bigg )\rightarrow0.
\label{stat-10a}
\end{equation}
In view of statement (\ref{stat-3b0}), but now with the function $a(s)=H_{i}%
\circ Q(s)(1-s)^{\gamma_{i}-1/2-\nu_{i}}$ (see Lemma \ref{lemma4.1} below), we have that the ratio in the parentheses is asymptotically bounded. Since $\nu_{i}>0$ and  $k=k_{n}\rightarrow\infty$, statement (\ref{stat-10a}) holds and thus statement (\ref{stat-9a}) follows.

We conclude the proof of Theorem \ref{theor-1a} by noting that the
integral on the right-hand side of equation (\ref{stat-9e}) can be made into
$\ell_{i,n}(1;B_{n})$. For this we need to check that
\[
{\frac{-1 }{(n/k)^{1/2}D_{i,n} }}\int_{0}^{1/n} (H_{i}\circ Q)'(s)
B_{n}(s) d\Psi_{i}(s) = o_{\mathbf{P}}(1),
\]
but latter statement holds because the integral is asymptotically bounded in probability and $(n/k)^{1/2}D_{i,n}\to\infty$ when $n\to\infty$. This concludes the proof of Theorem \ref{theor-1a}.
\end{proof}

\begin{lemma}\label{lemma4.1}
Statement (\ref{stat-3b0}) holds with the functions $a(s)=H_{i}\circ Q(s)(1-s)^{\gamma_{i}-1/2-\nu_i}$ and $b(s)=\int_{s}^{1}(1-t)^{-\gamma_{i}}d\Psi_{i}(t)$ assuming that the parameter
$\nu_{i}$ is either zero or (strictly) positive but sufficiently small.
\end{lemma}

\begin{proof}
To prove the lemma, we show that the limit
\begin{equation}
\lim_{\epsilon\downarrow0}{1\over \alpha(\epsilon)\beta(\epsilon)} \int_{\epsilon}^{1/2} \alpha
(s)d\beta(s)
\label{statem-this}
\end{equation}
is finite, where $\alpha(s)=H_{i}\circ Q(1-s) s^{\gamma_{i}-1/2-\nu_i}$ and $\beta(s)=\int_{1-s}^{1} (1-t)^{-\gamma_{i}}d\Psi_{i}(t)$. This reformulation of the problem helps us to connect the current proof with that of Lemma 1 in Necir and Meraghni (2010). Since $H_{i}\circ Q$ is differentiable and regularly varying at $1$ with the index $-\gamma_{i} <0$, the function $s\mapsto \alpha(s)$ is differentiable and regularly
varying at $0$ with the index $-\gamma_{i} + (\gamma_{i}-1/2-\nu_i)$.
Since the function $s\mapsto b(s)$ is regularly varying at $1$ with the index $\kappa> 0$ by an assumption of Theorem \ref{theor-1a}, the function $s\mapsto \beta(s)$ is regularly varying at $0$ with the index $\kappa$.
Consequently, the product function $s\mapsto \alpha(s)\beta(s)$ is regularly
varying at $0$ with the index $\kappa-1/2-\nu_i$, which is strictly negative for all sufficiently small $\nu_i$, and thus the product function converges to $+\infty$ when $s\downarrow0$. This fact and integration by parts formula imply that limit (\ref{statem-this}) is finite when the limit
\begin{equation}
\lim_{\epsilon\downarrow 0}{\frac{1 }{\alpha(\epsilon)\beta(\epsilon)}} \int_{\epsilon}^{1/2}\alpha'(s)\beta(s)ds
\label{statem-this-2}
\end{equation}
is such. Since $\epsilon \alpha'(\epsilon)=(-1/2)\alpha(\epsilon)(1+o(1))$ when $\epsilon\downarrow0$ by Karamata's representation (e.g., Proposition
B.1.9 (11) on page 367 of de Hann and Ferreira, 2006), limit (\ref{statem-this-2}) is finite when the limit
\begin{equation}
\lim_{\epsilon\downarrow0}{\frac{1 }{\epsilon\alpha'(\epsilon)\beta(\epsilon)}} \int_{\epsilon}^{1/2}\alpha'(s)\beta(s)ds
\label{stat-3b0ii}
\end{equation}
is such. The function $s\mapsto \alpha'(s)\beta(s)$ is regularly varying at $0$ with the index $(-1/2)-1+\kappa<0$, and thus Theorem B.1.5 on page 363 of de Haan and Ferreira (2006) implies the finiteness of limit (\ref{stat-3b0ii}). This proves the asymptotic boundedness of the right-hand side of equation (\ref{stat-3b0ii}) and  completes the proof of Lemma \ref{lemma4.1}.
\end{proof}

\begin{proof}[\textbf{Proof of Theorem \ref{theor-1ab}}]

We have
\begin{align}
\frac{\sqrt{n}(\widehat{L}_{i,n}(1)-L_{i,n}(1)}{(n/k)^{1/2}D_{i,n}}  &
={\frac{\sqrt{n}}{(n/k)^{1/2}D_{i,n}}}\int_{0}^{1-k/n}\left(  H_{i}\circ
Q_{n}(s)-H_{i}\circ Q(s)\right)  ds.\nonumber\\
&  ={\frac{\sqrt{n}}{(n/k)^{1/2}D_{i,n}}}\int_{1/n}^{1-k/n}\left(  H_{i}\circ
Q_{n}(s)-H_{i}\circ Q(s)\right)  ds+o_{\mathbf{P}}(1), \label{eq-start-1}%
\end{align}
where the second equality follows from statement (\ref{stat-1a}). Our
next step is based on the Vervaat process
\[
V_{n}(t)=\int_{0}^{t}\left(  H_{i}\circ Q_{n}(s)-H_{i}\circ Q(s)\right)
ds+\int_{-\infty}^{H_{i}\circ Q(t)}\left(  F_{n}\circ H_{i}^{-1}(x)-F\circ
H_{i}^{-1}(x)\right)  dx.
\]
We write the integral on the right-hand side of equation (\ref{eq-start-1})
as follows:
\begin{align}
\int_{1/n}^{1-k/n}
&\left(  H_{i}\circ Q_{n}(s)-H_{i}\circ Q(s)\right)  ds
\notag
\\
&=-\int_{H_{i}\circ Q(1/n)}^{H_{i}\circ Q(1-k/n)}\left(  F_{n}\circ H_{i}%
^{-1}(x)-F\circ H_{i}^{-1}(x)\right)  dx+V_{n}(1-k/n)-V_{n}(1/n)
\notag
\\
&=-\int_{H_{i}\circ Q(1/n)}^{H_{i}\circ Q(1-k/n)}e_{n}(F\circ H_{i}^{-1}(x))dx+V_{n}(1-k/n)-V_{n}(1/n),
\label{eq-0}
\end{align}
where $e_{n}(t)=\sqrt{n}\,(F_{n}\circ Q(t)-t)$ is the uniform empirical process.
It is known (see Zitikis, 1998; Davydov and Zitikis, 2003,
2004) that $V_{n}(t)$ is non-negative and does not exceed $-(F_{n}\circ
H_{i}^{-1}\circ H_{i}\circ Q(t))-t)(H_{i}\circ Q_{n}(t)-H_{i}\circ Q(t))$.
Since $F\circ H_{i}^{-1}$ is continuous by assumption, we therefore have the bound
\begin{equation}
\sqrt{n}\,V_{n}(t)\leq|e_{n}(t)||H_{i}\circ Q_{n}(t)-H_{i}\circ Q(t)|.
\label{eq-1}
\end{equation}
Consequently,
\begin{equation}
\frac{\sqrt{n}(\widehat{L}_{i,n}(1)-L_{i,n}(1)}{(n/k)^{1/2}D_{i,n}}%
=\frac{-1}{(n/k)^{1/2}D_{i,n}}\int_{H_{i}\circ Q(1/n)}^{H_{i}\circ Q(1-k/n)}e_{n}(F\circ H_{i}^{-1}(x))dx+O_{\mathbf{P}}(r_{n}^{*}+r_{n}^{**}),
\label{eq-4}
\end{equation}
where
\[
r_{n}^{*}={\frac{|e_{n}(1-k/n)||H_{i}\circ Q_{n}(1-k/n)-H_{i}\circ Q(1-k/n)|}{(n/k)^{1/2}D_{i,n}}}
\]
and
\[
r_{n}^{**}={\frac{|e_{n}(1/n)||H_{i}\circ Q_{n}(1/n)-H_{i}\circ Q(1/n)|}{(n/k)^{1/2} D_{i,n}}}.
\]
We have $r_{n}^{**}=o_{\mathbf{P}}(1)$ because
$|e_{n}(1/n)||H_{i}\circ Q_{n}(1/n)-H_{i}\circ Q(1/n)|=O_{\mathbf{P}}(1)$ and
$(n/k)^{1/2}D_{i,n}\rightarrow\infty$. To show that $r_{n}^{*}=o_{\mathbf{P}}(1)$, we shall next check that
\begin{equation}
\frac{|e_{n}(1-k/n)|}{(k/n)^{1/2}}\bigg |{\frac{H_{i}\circ Q_{n}(1-k/n)}%
{H_{i}\circ Q(1-k/n)}}-1\bigg |=o_{\mathbf{P}}(1). \label{eq-10}%
\end{equation}
For this, we write the bound
\begin{equation}
\frac{|e_{n}(1-k/n)|}{(k/n)^{1/2}}\leq\frac{|e_{n}(1-k/n)-B_{n}(1-k/n)|}%
{(k/n)^{1/2}}+\frac{|B_{n}(1-k/n)|}{(k/n)^{1/2}}. \label{eq-12}%
\end{equation}
The first summand on the right-hand side of bound (\ref{eq-12}) is of the
order $O_{\mathbf{P}}(1)$ due to Corollary 2.1 on page 48 of Cs\"{o}rg\H{o}
et al. (1986), which states that on an appropriate probability space and for any $0\leq\nu<1/4$, we have that
\begin{equation}
\sup_{1/n\leq s\leq1-1/n}n^{\nu}{\frac{|e_{n}(s)-B_{n}(s)|}{s^{1/2-\nu
}(1-s)^{1/2-\nu}}}=O_{\mathbf{P}}(1).
\label{str-approx}
\end{equation}
Statement (\ref{str-approx}) with $\nu=0$ implies that the first summand on the right-hand side of bound (\ref{eq-12}) is of the order $O_{\mathbf{P}}(1)$.

\begin{note}\rm
Compare statements (\ref{str-approx}) and (\ref{stat-cchm}): both of them use
same Brownian bridges $B_{n}$ and probability spaces as specified in
Cs\"{o}rg\H{o} et al. (1986).
\end{note}

The second summand on the right-hand side of bound (\ref{eq-12}) is also of the order $O_{\mathbf{P}}(1)$ due to a statement on page 49 of Cs\"{o}rg\H{o} et al. (1986): see the displayed bound there just below statement (2.39). Hence, statement (\ref{eq-10}) follows from
\begin{equation}
{\frac{H_{i}\circ Q_{n}(1-k/n)}{H_{i}\circ Q(1-k/n)}}=1+o_{\mathbf{P}}(1),
\label{eq-13}
\end{equation}
which we shall establish in Lemma \ref{lemma4.2} below.
In summary, we have proved that
\begin{equation}
\frac{\sqrt{n}(\widehat{L}_{i,n}(1)-L_{i,n}(1)}{(n/k)^{1/2}D_{i,n}}
=\frac{-1}{(n/k)^{1/2}D_{i,n}}\int_{H_{i}\circ Q(1/n)}^{H_{i}\circ Q(1-k/n)}e_{n}(F\circ H_{i}^{-1}(x))dx+o_{\mathbf{P}}(1).
\label{eq-4oo}
\end{equation}

We next replace the empirical process $e_{n}$ n the right-hand side of equation
(\ref{eq-4oo}) by an appropriate Brownian bridge $B_{n}$ with an error term of the order $o_{\mathbf{P}}(1)$. Statement (\ref{str-approx})
implies that the replacement of $e_{n}$ by $B_{n}$ is possible with an error  $o_{\mathbf{P}}(1)$, provided that the quantity
\[
\frac{1}{n^{\nu}(n/k)^{1/2}D_{i,n}}\int_{H_{i}\circ Q(1/n)}^{H_{i}\circ Q(1-k/n)}F\circ H_{i}%
^{-1}(x)^{1/2-\nu}(1-F\circ H_{i}^{-1}(x))^{1/2-\nu}dx
\]
converges to $0$ when $n\rightarrow\infty$. To show this, we first change the
variable of integration, then integrate by parts, and in this way reduce our task to
showing that, when $n\rightarrow\infty$,
\begin{equation}
\frac{1}{n^{\nu}(n/k)^{1/2}D_{i,n}}s^{1/2-\nu}(1-s)^{1/2-\nu}H_{i}\circ Q(s)|_{1/n}^{1-k/n} \to 0
\label{eq-re1}
\end{equation}
and
\begin{equation}
\frac{1}{n^{\nu}(n/k)^{1/2}D_{i,n}}\int_{1/n}^{1-k/n}s^{-1/2-\nu}
(1-s)^{-1/2-\nu}H_{i}\circ Q(s)ds \to 0.
\label{eq-re2}
\end{equation}

To prove statement (\ref{eq-re1}), we first recall that we are
currently dealing with the case $\Psi_i(s)=s$, and then replace $D_{i,n}$ by
$(k/n)H_{i}\circ Q(1-k/n)$ since, up to a constant, the two quantities are asymptotically equivalent. We have that
\begin{multline}
\frac{1}{n^{\nu}(k/n)^{1/2}H_{i}\circ Q(1-k/n)}s^{1/2-\nu}(1-s)^{1/2-\nu}H_{i}\circ Q(s)|_{1/n}^{1-k/n}
\\
=O\bigg (\frac{1}{k^{\nu}}\bigg )+O\bigg (\frac{H_{i}\circ Q(1/n)}{n^{1/2}(n/k)^{1/2}D_{i,n}}\bigg ),
\label{eq-8}
\end{multline}
which converges to $0$ because $k=k_{n}\rightarrow\infty$ and $(n/k)^{1/2}%
D_{i,n}\rightarrow\infty$ when $n\rightarrow\infty$.

To prove statement (\ref{eq-re2}), we use the notation
$\phi(u)=H_{i}\circ Q(1-u)/u^{1/2+\nu}$ and write
\begin{equation}
\frac{1}{n^{\nu}(k/n)^{1/2}H_{i}\circ Q(1-k/n)}
\int_{1/n}^{1-k/n}(1-s)^{-1/2-\nu}H_{i}\circ Q(s)ds
={\frac{1}{k^{\nu}}}{\frac{\int_{k/n}^{1-1/n}\phi(s)ds}{(k/n)\phi(k/n)}}.
\label{eq-9}
\end{equation}
The right-hand side converges to $0$ when $n\rightarrow\infty$ as seen from
Result 1 in Appendix of Necir and Meraghni (2009).

In summary, we have proved that
\begin{align}
\frac{\sqrt{n}(\widehat{L}_{i,n}(1)-L_{i,n}(1)}{(n/k)^{1/2}D_{i,n}}
&=\frac{-1}{(n/k)^{1/2}D_{i,n}}\int_{H_{i}\circ Q(1/n)}^{H_{i}\circ Q(1-k/n)}B_{n}\circ F\circ H_{i}^{-1}(x)dx
+o_{\mathbf{P}}(1)
\notag
\\
&  ={\frac{-1}{(n/k)^{1/2}D_{i,n}}}\int_{1/n}^{1-k/n}B_{n}(s)dH_{i}\circ
Q(s)+o_{\mathbf{P}}(1)
\notag
\\
&  ={\frac{-1}{(n/k)^{1/2}D_{i,n}}}\int_{0}^{1-k/n}B_{n}(s)dH_{i}\circ
Q(s)+o_{\mathbf{P}}(1).
\label{eq-6aa}
\end{align}
Upon noticing that the right-hand side of equation (\ref{eq-6aa}) is equal to $\ell
_{i,n}(1;B_{n})+o_{\mathbf{P}}(1)$, we complete the proof of Theorem \ref{theor-1ab}.
\end{proof}

\begin{lemma}\label{lemma4.2}
Statement (\ref{eq-13}) holds.
\end{lemma}

\begin{proof}
The distribution of $H_{i}\circ Q_{n}(1-k/n) $ is same as that
of $H_{i}\circ Q\left(  E_{n}^{-1}(1-k/n) \right)  $, where $E_{n}^{-1}$ is
the uniform empirical quantile function. Furthermore, the two processes $\{
1-E_{n}^{-1}( 1-s) ,\ 0\leq s\leq1\}$ and $\{ E_{n}^{-1}( s) ,\ 0\leq
s\leq1\}$ are equal in distribution. Hence, statement (\ref{eq-13})
is equivalent to
\begin{equation}
{\frac{H_{i}\circ Q\left(  1-E_{n}^{-1}(k/n)\right)  }{H_{i}\circ Q(1-k/n)}}
=1+o_{\mathbf{P}}(1).
\label{eq-13a}
\end{equation}
From the Glivenko-Cantelli theorem, $E_{n}^{-1}(k/n)-k/n\to0$ almost surely.
This also implies that $E_{n}^{-1}(k/n)\to0$ since $k/n\to0$. Moreover, by
Theorem 0 and Remark 1 of Wellner (1978), we have
$\sup_{1/n\leq s\leq1}s^{-1}\left\vert E_{n}^{-1}(s)-s\right\vert =o_{\mathbf{P}}(1)$. Hence,
\begin{equation}
(n/k)E_{n}^{-1}(k/n)=1+o_{\mathbf{P}}(1).
\label{eq-13b}
\end{equation}
Since $s\mapsto H_{i}\circ Q\left(  1-s\right)  $ is slowly varying at zero,
using Potter's inequality (see, e.g., the $5^{\text{th}}$ assertion of
Proposition B.1.9 on page 367 of de Haan and Ferreira (2006)) we have
\begin{equation}
\frac{H_{i}\circ Q\left(  1-E_{n}^{-1}(k/n)\right)  }{H_{i}\circ Q(1-k/n)
}=\left(  1+o_{\mathbf{P}}(1)\right)  \left(  (n/k)E_{n}^{-1}(k/n)\right)
^{-\gamma_{i}\pm\theta}
\label{eq-13c}
\end{equation}
for any $\theta\in(0,\gamma_{i} )$. In view of (\ref{eq-13b}), the right-hand
side of equation (\ref{eq-13c}) is equal to $1+o_{\mathbf{P}}(1)$. This
implies statement (\ref{eq-13a}) and thus, in turn, statement (\ref{eq-13}).
\end{proof}

\begin{proof}[\textbf{Proof of Theorem \ref{theor-1b}}]

Let $Y$ be a random variable with the cdf $G(z)=1-1/z$, $z\geq1$. Furthermore, let $\mathbb{U}_{i}(z)=H_{i}\circ Q(1-1/z)$. Since $G(Y)$ is uniform on the interval
$[0,1]$, the random variables $H_{i}(X)$ and $\mathbb{U}_{i}(Y)$ are
equal in distribution. We have the representations
\[
\widehat{L}_{i,n}(2) =\mathbb{U}_{i}(Y_{n-k:n}) \left(  \frac{k}%
{n}\right)  ^{\widehat{\gamma}_{i}} \int_{1-k/n}^{1}(1-s)^{-\widehat{\gamma
}_{i}}d\Psi_{i}(s)
\]
and
\[
D_{i,n} =\mathbb{U}_{i}(n/k)\left(  \frac{k}{n}\right)  ^{\gamma_{i}}
\int_{1-k/n}^{1} (1-s)^{-\gamma_{i}}d\Psi_{i}(s).
\]
Consequently,
\begin{align}
\frac{\sqrt{n}\,( \widehat{L}_{i,n}(2)-L_{i,n}(2)) }{(n/k)^{1/2}D_{i,n}}
= &
\sqrt{k}\,\left(  \frac{\mathbb{U}_{i}(Y_{n-k:n}) }{\mathbb{U}_{i}(n/k)
}-{\frac{\int_{0}^{1}H_{i}\circ Q(1-ks/n) d\Psi
_{i}(1-ks/n) }{H_{i}\circ Q(1-k/n) \int_{0}^{1} s^{-\gamma_{i}}d\Psi
_{i}(1-ks/n)}} \right)
\notag
\\
 =  &
\sqrt{k}\,\left(  \frac{\mathbb{U}_{i}(Y_{n-k:n}) }{\mathbb{U}_{i}(n/k)
}-\left(  \frac{Y_{n-k:n}}{n/k}\right)  ^{\gamma_{i}}\right)  {\frac{\left(
k/n\right)  ^{\widehat{\gamma}_{i}} \int_{1-k/n}^{1}(1-s)^{-\widehat
{\gamma}_{i}}d\Psi_{i}(s) }{\left(  k/n\right)  ^{\gamma_{i}}
\int_{1-k/n}^{1} (1-s)^{-\gamma_{i}}d\Psi_{i}(s) }}\nonumber\\
&  +\sqrt{k}\,\left(  \left(  \frac{Y_{n-k:n}}{n/k}\right)  ^{\gamma_{i}%
}-1\right)  {\frac{\left(  k/n\right)  ^{\widehat{\gamma}_{i}}
\int_{1-k/n}^{1}(1-s)^{-\widehat{\gamma}_{i}}d\Psi_{i}(s) }{(k/n)  ^{\gamma_{i}} \int_{1-k/n}^{1} (1-s)^{-\gamma_{i}}d\Psi_{i}(s)
}}\nonumber\\
&  + \sqrt{k}\, \left(  {\frac{\left(  k/n\right)  ^{\widehat{\gamma
}_{i}} \int_{1-k/n}^{1}(1-s)^{-\widehat{\gamma}_{i}}d\Psi_{i}(s) }{\left(
k/n\right)  ^{\gamma_{i}} \int_{1-k/n}^{1} (1-s)^{-\gamma_{i}}%
d\Psi_{i}(s) }} -1 \right) \nonumber\\
&  +\sqrt{k}\, \left(  1 -{\frac{\int_{0}^{1}H_{i}\circ Q(1-ks/n) d\Psi
_{i}(1-ks/n) }{H_{i}\circ Q(1-k/n) \int_{0}^{1} s^{-\gamma_{i}}d\Psi
_{i}(1-ks/n)}} \right)  .
\label{equa-1}
\end{align}
We shall next investigate the asymptotic behaviour of the four terms on the right-hand
side of equation (\ref{equa-1}).

\subsubsection*{First term on the RHS of equation (\ref{equa-1})}

We shall show that this term is of the order $o_{\mathbf{P}}(1)$, which follows from the two statements:
\begin{equation}
\sqrt{k}\left(  {\frac{\mathbb{U}_{i}(Y_{n-k:n})}{\mathbb{U}_{i}(n/k)}}-\left(  \frac{Y_{n-k:n}}{n/k}\right)^{\gamma_{i}}\right)  =o_{\mathbf{P}}(1)
\label{equa-10}
\end{equation}
and
\begin{equation}
{\frac{\left(  k/n\right)  ^{\widehat{\gamma}_{i}}\int_{1-k/n}%
^{1}(1-s)^{-\widehat{\gamma}_{i}}d\Psi_{i}(s)}{\left(  k/n\right)
^{\gamma_{i}}\int_{1-k/n}^{1}(1-s)^{-\gamma_{i}}d\Psi_{i}(s)}}=1+o_{\mathbf{P}}(1).
\label{equa-2}
\end{equation}
To prove statement (\ref{equa-10}), we note that since $H_{i}\circ Q$ is
regularly varying at $1$ with the index $(-\gamma_{i})$, the statement can be verified using analogous arguments as those used for proving  statement $T_{n1}=o_{\mathbf{P}}(1)$ by Necir et al. (2007). To check
statement (\ref{equa-2}), we note that we shall later establish a
non-degenerate limiting distribution of the third term on the right-hand side
of equation (\ref{equa-1}), which is a stronger statement than (\ref{equa-2}). With these notes we conclude our proof that the first term on the right-hand side of equation (\ref{equa-1}) is of the order $o_{\mathbf{P}}(1)$.

\subsubsection*{Second term on the RHS of equation (\ref{equa-1})}

Following the arguments on page 158 of Necir et al. (2007) and making notational adjustments, we have that
\begin{equation}
\sqrt{k}\,\left(  \left(  \frac{Y_{n-k:n}}{n/k}\right)  ^{\gamma_{i}%
}-1\right)  =\gamma_{i}\sqrt{{\frac{n}{k}}}\,\sqrt{n}\,\vartheta
_{n}(1-k/n)+o_{\mathbf{P}}(1), \label{equa-12-0}%
\end{equation}
where the (un-normalized) uniform quantile process $\vartheta_{n}$ is defined by equation (\ref{quant-process}). Utilizing weak approximation results of Cs\"{o}rg\H{o}
et al. (1986), we have Brownian bridges $B_n$ on an appropriate probability space such that
\begin{equation}
\sqrt{k}\,\left(  \left(  \frac{Y_{n-k:n}}{n/k}\right)  ^{\gamma_{i}%
}-1\right)  =-\gamma_{i}\sqrt{{\frac{n}{k}}}\,B_{n}(1-k/n)+o_{\mathbf{P}}(1).
\label{equa-12}
\end{equation}
This and also statement (\ref{equa-2}) imply that the second term on the right-hand side of equation (\ref{equa-1}) has the same asymptotic behaviour as the right-hand side of equation (\ref{equa-12}).

\subsubsection*{Third term on the RHS of equation (\ref{equa-1})}

We rewrite the term as follows:
\begin{align}
\sqrt{k}\,\left(  {\frac{\left(  k/n\right)  ^{\widehat{\gamma}_{i}%
}\int_{1-k/n}^{1}(1-s)^{-\widehat{\gamma}_{i}}d\Psi_{i}(s)}{\left( k/n\right)  ^{\gamma_{i}}\int_{1-k/n}^{1}(1-s)^{-\gamma_{i}}d\Psi_{i}(s)}}-1\right)
&  =\sqrt{k}\,\left(  {\frac{\int_{0}^{1}s^{-\widehat{\gamma}_{i}}
d\Psi_{i}(1-ks/n)}{\int_{0}^{1}s^{-\gamma_{i}}d\Psi_{i}(1-ks/n)}}-1\right)
\nonumber
\\
&  =\sqrt{k}\,{\frac{\int_{0}^{1}(  s^{-\widehat{\gamma}_{i}}-s^{-\gamma_{i}})  d\Psi_{i}(1-ks/n)}{\int_{0}^{1}s^{-\gamma_{i}}d\Psi_{i}(1-ks/n)}}.
\label{equa-12uu}
\end{align}
Theorems 2.3 and 2.4 of Cs\"{o}rg\H{o} and Mason (1985) imply that there are Brownian bridges $B_{n}$ on a probability space specified by Cs\"{o}rg\H{o} et al. (1986) such that
\begin{equation}
{\frac{\sqrt{k}\left(  \widehat{\gamma}_{i}-\gamma_{i}\right)  }{\gamma_{i}}}
=\sqrt{{\frac{n}{k}}}\,B_{n}(1-k/n)-\sqrt{{\frac{n}{k}}}\,
\int_{1-k/n}^{1}\frac{B_{n}(s)}{1-s}ds+o_{\mathbf{P}}(1).
\label{weak-approx-1}
\end{equation}
Hence, in particular,  $\widehat{\gamma}$ is a consistent estimator of $\gamma$, and so we have
\begin{multline}
\int_{0}^{1}\left(  s^{-\widehat{\gamma}_{i}}-s^{-\gamma_{i}}\right)
d\Psi_{i}(1-ks/n)=-\left(  \widehat{\gamma}_{i}-\gamma_{i}\right)  \int
_{0}^{1}s^{-\gamma_{i}}\log(s)d\Psi_{i}(1-ks/n)\\
+o_{\mathbf{P}}(1)\left\vert \widehat{\gamma}_{i}-\gamma_{i}\right\vert
\int_{0}^{1}s^{-\gamma_{i}-\delta}d\Psi_{i}(1-ks/n),
\label{equa-3}
\end{multline}
where $\delta>0$ can be any fixed number. Furthermore, in view of statement (\ref{weak-approx-1}) we have that $\sqrt{k}(\widehat{\gamma}_{i}-\gamma)=O_{\mathbf{P}}(1)$. This result,
assumptions (\ref{con-0b}) and (\ref{equa-10c}), and asymptotic representation  (\ref{equa-3}) imply that the right-hand side of equation (\ref{equa-12uu}) is equal to $-c_{i}\sqrt{k}(\widehat{\gamma}_{i}-\gamma_{i})+o_{\mathbf{P}}(1)$. Utilizing result (\ref{weak-approx-1}) in its full generality, we conclude that
\begin{multline}
\sqrt{k}\,\left(  {\frac{(k/n)^{\widehat{\gamma}_{i}}
\int_{1-k/n}^{1}(1-s)^{-\widehat{\gamma}_{i}}d\Psi_{i}(s)}{(k/n)^{\gamma_{i}}\int_{1-k/n}^{1}(1-s)^{-\gamma_{i}}d\Psi_{i}(s)}}-1\right)  =-\gamma_{i}c_{i}\sqrt{{\frac{n}{k}}}\,B_{n}(1-k/n)
\\
+\gamma_{i}c_{i}\sqrt{{\frac{n}{k}}}\,\int_{1-k/n}^{1}\frac{B_{n}(s)}{1-s}ds
+o_{\mathbf{P}}(1).
\label{equa-15}
\end{multline}
This determines the asymptotic behaviour of the third term on the right-hand side of equation (\ref{equa-1}) and also establishes the earlier noted statement (\ref{equa-2}).

\subsubsection*{Fourth term on the RHS of equation (\ref{equa-1})}

We rewrite the term as follows:
\begin{multline}
\sqrt{k}\,\left(  1-{\frac{\int_{0}^{1}H_{i}\circ Q(1-ks/n)d\Psi_{i}%
(1-ks/n)}{H_{i}\circ Q(1-k/n)\int_{0}^{1}s^{-\gamma_{i}}d\Psi_{i}(1-ks/n)}%
}\right) \\
={\frac{-\sqrt{k}}{\int_{0}^{1}s^{-\gamma_{i}}d\Psi_{i}(1-ks/n)}}
\int_{0}^{1}\left(  {\frac{H_{i}\circ Q(1-ks/n)}{H_{i}\circ
Q(1-k/n)}}-s^{-\gamma_{i}}\right)  d\Psi_{i}(1-ks/n).
\label{equa-10a}
\end{multline}
From Theorem 2.3.9 on page 48 of de Haan and Ferreira (2006) we have a function $\widetilde{A}_{i}$ with the property $\widetilde{A}_{i}(x)\sim A_{i}(x) $ when $x\rightarrow \infty $ such that for any $\delta>0$ there exists $\epsilon_{0,i}\in (0,1) $ such that, whenever $\epsilon s\leq \epsilon_{0,i}$,
\begin{equation}
\left\vert {\frac{1}{\widetilde{A}_{i}(1/\epsilon)}}\left(  {\frac{H_{i}\circ
Q(1-\epsilon s)}{H_{i}\circ Q(1-\epsilon)}}-s^{-\gamma_{i}}\right)  -s^{-\gamma_{i}}%
\frac{s^{-\omega_{i}}-1}{\omega_{i}}\right\vert \leq\delta s^{-\gamma_{i}-\omega
_{i}-\delta}.
\label{equa-11z}
\end{equation}
Upon recalling condition (\ref{con-0a}), we have $\sqrt{k}\,\widetilde{A}_{i}(n/k)  \rightarrow b_{i}$. Inequality (\ref{equa-11z}) implies that for all sufficiently large $n$ the right-hand side of equation (\ref{equa-10a}) is equal to
\begin{equation}
\bigg ( {\frac{-b_{i}}{\omega_{i}}} \bigg ) {\frac{\int_{0}^{1}s^{-\gamma_{i}}(s^{-\omega_{i}%
}-1)d\Psi_{i}(1-ks/n)}{\int_{0}^{1}s^{-\gamma_{i}}d\Psi_{i}(1-ks/n)}}+o(1)
\frac{\int_{0}^{1}s^{-\gamma_{i}-\omega
_{i}-\delta}d\Psi_{i}(1-ks/n)}{\int_{0}^{1}s^{-\gamma_{i}}d\Psi_{i}(1-ks/n)}.
\label{equa-11d}
\end{equation}
In view of assumptions (\ref{con-0c}) and (\ref{equa-10c}), we conclude that quantity (\ref{equa-11d}) and thus the fourth term on the right-hand side of equation (\ref{equa-1}) are of the order $-b_{i}d_{i}/\omega_{i}+o(1)$.

Collecting the above established asymptotic properties of the four terms on
the right-hand side of equation (\ref{equa-1}), we finish the proof of
Theorem \ref{theor-1b}.
\end{proof}

\section*{\bf Acknowledgments}

The research has been partially supported by
the Natural Sciences and Engineering Research Council (NSERC) of Canada.

\section*{\bf References}
\def\hang{\hangindent=\parindent\noindent}

\hang {Beirlant, J., Groegebeur, Y., Teugeles, J. (2004).}
\textit{Statistics of Extremes, Theory and Applications.}
Wiley, New York.

\hang
Brazauskas, V., Jones, B.L., Zitikis, R.\, (2007).
Robustification and performance evaluation of empirical risk measures
and other vector-valued estimators.
\textit{Metron -- International Journal of Statistics}, 65 (2), 175--199.

\hang
Brazauskas, V., Jones, B.L., Zitikis, R.\, (2009).
Robust fitting of claim severity distributions
and the method of trimmed moments.
\textit{Journal of Statistical Planning and Inference},
139 (6), 2028--2043.

\hang
Brazauskas, V., Jones, B.L., Puri, M.L., Zitikis, R.\, (2008).
Estimating conditional tail expectation with
actuarial applications in view.
\textit{Journal of Statistical Planning and Inference},
138 (11, Special Issue in Honor of Junjiro Ogawa:
Design of Experiments, Multivariate Analysis
and Statistical Inference), 3590--3604.

\hang
Castillo, E., Hadi, A. S., Balakrishnan, N., Sarabia, J. M. (2005).
\textit{Extreme Value and Related Models with Applications
in Engineering and Science}. Wiley, Hoboken, NJ.

\hang {Cs\"{o}rg\H{o}, M., Cs\"{o}rg\H{o}, S.,
Horv\'{a}th, L., Mason, D. M. (1986).} Weighted empirical and quantile
processes. \textit{Annals of Probability} 14, 31--85.

\hang {Davydov, Y., Zitikis, R. (2003).}
Generalized Lorenz curves and convexifications of stochastic processes
\textit{Journal of Applied Probability} 40, 906--925.

\hang {Davydov, Y., Zitikis, R. (2004).}
Convex rearrangements of random elements. In: \textit{Asymptotic Methods in Stochastics} (Eds.: L. Horv\'ath and B. Szyszkowicz), pp.~141--171. American Mathematical Society, Providence, RI.

\hang {Denneberg, D. (1994).}
\textit{Non-additive Measure and Integral}. Dordrecht, Kluwer.

\hang  Denuit, M., Dhaene, J., Goovaerts, M.J., Kaas, R. (2005).
\textit{Actuarial Theory for Dependent Risk: Measures, Orders and Models}.
Wiley, New York.

\hang {Furman, E., Zitikis, R. (2008).}
Weighted premium calculation principles. \textit{Insurance: Mathematics and
Economics} 42, 459--465.

\hang Furman, E., Zitikis, R. (2009).
Weighted pricing functionals with applications
to insurance: an overview.
\textit{North American Actuarial Journal} 13, 483--496.

\hang
Greselin, F., Pasquazzi, L.\, and Zitikis, R.\, (2010).
Zenga's new index of economic inequality,
its estimation, and an analysis of incomes in Italy.
\textit{Journal of Probability and Statistics}
(Special Issue on ``Actuarial and Financial Risks:
Models, Statistical Inference, and Case Studies''),
2010, Article ID 718905, 26 pages.

\hang
Greselin, F., Puri, M.L.\, and Zitikis, R.\, (2009).
$L$-functions, processes, and statistics
in measuring economic inequality and actuarial risks.
\textit{Statistics and Its Interface},
2 (2, Festschrift for Professor Joseph L. Gastwirth), 227--245.

\hang {de Haan, L., Ferreira, A. (2006).}
\textit{Extreme Value Theory: An introduction.} Springer, New York.

\hang {Hill, B.M. (1975)}. A simple approach
to inference about the tail of a distribution.
\textit{Annals of Statistics} 3, 1136-1174.

\hang
Jones, B.L., Zitikis, R.\, (2003).
Empirical estimation of risk measures and related quantities.
\textit{North American Actuarial Journal}, 7 (4), 44--54.
Discussion by V.~Brazauskas and T.~Kaiser
in \textit{North American Actuarial Journal}, 8 (3; 2004) 114--117.
Authors' Reply
in \textit{North American Actuarial Journal}, 8 (3; 2004) 117--118.

\hang {Necir, A., Meraghni, D. (2009).}
Empirical estimation of the proportional hazard premium for heavy-tailed claim
amounts. \textit{Insurance: Mathematics and Economics} 45, 49-58.

\hang {Necir, A., Meraghni, D. (2010).}
Estimating $L$-functionals for heavy-tailed distributions and applications.
\textit{Journal of Probability and Statistics}\ 2010, ID 707146.

\hang {Necir, A., Rassoul, A., Zitikis, R.
(2010).} Estimating the conditional tail expectation in the case of
heavy-tailed losses. \textit{Journal of Probability and Statistics} 2010, ID 956838.

\hang {Necir, A., Meraghni, D., Meddi, F. (2007).} Statistical estimate of the proportional hazard premium of
loss. \textit{Scandinavian Actuarial Journal} 3, 147--161.

\hang {Peng, L. (2001).} Estimating the mean
of a heavy tailed distribution. \textit{Statistics and Probability Letters}
52, 255--264.

\hang {Shorack, G.R., Wellner, J.A. (1986).}
\textit{Empirical Processes with Applications to Statistics.}
Wiley, New York.

\hang {Wang, S. (1995).} Insurance pricing and
increased limits ratemaking by proportional hazards transforms.
\textit{Insurance: Mathematics and Economics} 17, 43--54.

\hang {Wang, S. (1998).} An actuarial index of
the right-tail isk. \textit{North American Actuarial Journal} 2(2), 88--101.

\hang {Weissman, I., (1978).} Estimation of
parameters and large quantiles based on the $k$ largest observations.
\textit{Journal of American Statistical Association} 73, 812-815.

\hang
{Zenga M.} (1987).
Il contributo degli italiani allo studio della concentrazione.
In: \textit{La distribuzione personale del reddito: problemi di formazione,
di ripartizione e di misurazione} (Ed.: M. Zenga).
Vita e Pensiero, Milano.

\hang
{Zenga, M.} (2007).
Inequality curve and inequality index based
on the ratios between lower and upper arithmetic means.
\textit{Statistica \& Applicazioni} 5, 3--27.

\hang {Zitikis, R. (1998).} The Vervaat
process. In \textit{Asymptotic Methods in Probability and
Statistics} (Ed.: B. Szyszkowicz), pp.~667--694. Amsterdam: North-Holland.

\end{document}